\documentclass{amsart}

\usepackage{amssymb}   
\usepackage[all,cmtip]{xy}   
\usepackage{hyperref}       

\usepackage{color}

\newcommand{\Z}{\mathbb{Z}}               
\newcommand{\Q}{\mathbb{Q}}              

\newcommand{\Gr}{\mathrm{Gr}}      
\newcommand{\frakp}{\mathfrak{p}}   

\newcommand{\al}{\alpha}     
\newcommand{\be}{\beta}
\newcommand{\ep}{\epsilon}  
\newcommand{\la}{\lambda}   
\newcommand{\om}{\omega}   
\newcommand{\de}{\delta}  
\newcommand{\ka}{\kappa}
\newcommand{\Ga}{\Gamma}
\newcommand{\ga}{\gamma}

\newcommand{\pair}[1]{\langle #1\rangle} 

\newcommand{\Up}{\mathfrak{h}^*}  
\newcommand{\Upd}{\mathfrak{h}}   
\newcommand{\cl}{\Lambda}            
\newcommand{\rl}{\cl_r}                    

\newcommand{\ro}{\Phi}  
\newcommand{\sroot}{\Pi}  
\newcommand{\proot}{\ro_+} 
\newcommand{\nroot}{\ro_-}  
\newcommand{\Car}{A}     
\newcommand{\sdroot}{\sroot^\vee}  
\newcommand{\len}[1]{\ro_{#1}} 

\newcommand{\nila}{\mathcal{N}}

\newcommand{\lbr}{[\hspace{-1.5pt}[}  
\newcommand{\rbr}{]\hspace{-1.5pt}]}  
\newcommand{\FGR}[3]{#1 \lbr #2 \rbr_{#3}} 
\newcommand{\RcF}{\FGR{R}{\cl}{F}} 
\newcommand{\IF}{\mathcal{I}_F}    

\newcommand{\QW}{Q_W}

\newcommand{\HF}{\mathbf{H}_F}

\newcommand{\TR}{{\triangle}}

\newcommand{\DF}{{\mathbf{D}_F}}

\newcommand{\Dem}{\Delta}                  
\newcommand{\cC}{C}          
\newcommand{\oC}{Y}        
\newcommand{\DcF}{\mathbf{D}_{F}}    
\newcommand{\Tr}{\triangle}

\newcommand{\DcFa}{\mathbf{D}_a}
              
\newcommand{\DcFd}{\mathbf{D}_F^*} 

\newcommand{\unit}{\mathbf{1}}   
\DeclareMathOperator{\Hom}{\mathrm{Hom}}
\DeclareMathOperator{\rank}{\mathrm{rank}}
\DeclareMathOperator{\End}{\mathrm{End}}

\DeclareMathOperator{\Spec}{\mathrm{Spec}}
\theoremstyle{plain}
\newtheorem{theo}{Theorem}[section]
\newtheorem{prop}[theo]{Proposition}
\newtheorem{lem}[theo]{Lemma}
\newtheorem{cor}[theo]{Corollary}
\theoremstyle{definition}
\newtheorem{defi}[theo]{Definition}
\newtheorem{rem}[theo]{Remark}
\newtheorem{example}[theo]{Example}
\numberwithin{equation}{section}   

\begin{document}



\title{Formal affine Demazure and Hecke algebras of Kac-Moody root systems}

\author{Baptiste Calm\`es}
\author{Kirill Zainoulline}
\author{Changlong Zhong}
\thanks{ B. C. is support by the French Agence Nationale de la Recherche (ANR) under reference ANR-12-BL01-0005. K. Z. was partially supported by the NSERC Discovery Grant RGPIN-2015-04469. C. Z. was supported by PIMS and NSERC grants of S. Gille and V. Chernousov in 2014, and would like to thank S. Kumar, T. Zhang and G. Zhao for helpful discussions, and to the Max Planck Institute for Mathematics for hospitality during the visit in May, 2014.}
\maketitle

\begin{abstract}
We define the formal affine Demazure algebra and formal affine Hecke algebra associated to a Kac-Moody root systems. We prove the structure theorems of these algebras, hence, extending several results and constructions (presentation in terms of generators and relations, coproduct and product structures, filtration by codimension of Bott-Samelson classes, root polynomials and multiplication formulas) that were previously known for finite root systems. 
\end{abstract}

\section{Introduction}

The nil-Hecke algebra and $0$-Hecke algebra were constructed by Kostant-Kumar in \cite{KK86, KK90} in the study of equivariant singular cohomology and equivariant $K$-theory of Kac-Moody flag varieties. These algebras are  generated by the Demazure operators (also known as BGG operators or divided difference operators) satisfying braid relation; they act on their duals by the convolution action of equivariant singular cohomology (resp. equivariant $K$-theory) on itself. 

In \cite{HMSZ, CZZ1, CZZ2, CZZ3} these constructions were extended to an arbitrary algebraic oriented cohomology theory (corresponding
to a one-dimensional commutative formal group law $F$) of a usual flag variety, i.e. for finite root systems. The resulting algebra, called the formal affine Demazure and denoted $\DcF$, 
satisfies a twisted braid relation instead of the usual one and its dual is isomorphic to the respective oriented cohomology ring.
In particular, the nil-Hecke algebra and the $0$-Hecke algebra of Kostant-Kumar correspond to the additive and the multiplicative formal group law $F$ respectively. 

In the present paper,  we extend the previous constructions to an arbitrary Kac-Moody root system (determined by a generalized Cartan matrix), hence, completing the picture. This was originally a question raised by Shrawan Kumar. Note that the key result in this setup is the so-called structure theorem, which describes $\DcF$ by its polynomial representation. To prove it in the additive and multiplicative case, Kostant-Kumar essentially relied on the study of zeros and poles of such polynomial functions. For a general $F$, polynomials turn into formal power series, so we no longer can look at zeros or poles. To overcome this difficulty and to prove the generalized version of the structure theorem (Theorem~\ref{thm:mainDem}) we use the language of prime ideals and some general position arguments (see Section~\ref{sec:genpos}).

Having the structure theorem in hands we extend most of the results and properties concerning $\DcF$
following the same proofs as in the finite case.
For instance, we provide the presentation of $\DcF$ in terms of generators and relations (Corollary~\ref{cor:present}), hence, generalizing the result of \cite{Le}. We introduce the coproduct structure on $\DcF$ and, hence, the product on its dual $\DcFd$, hence, extending \cite{CZZ1}.
We define filtrations on $\DF$ and on $\DF^*$, and show that the corresponding associated graded rings are isomorphic to the nil-Hecke algebra and its dual, respectively (Proposition~\ref{prop:filtration}). 
In Section~\ref{sec:formhecke} we introduce and study the formal affine Hecke algebra $\HF$, extending definitions of \cite{HMSZ} and \cite{ZZ}. In the last section, we discuss formal root polynomials and various multiplication formulas, hence, extending some results of \cite{LZ}.


\section{Preliminaries}

\subsection{Generalized Cartan matrices and root data}
We introduce notation and list basic properties of the root datum associated to a generalized Cartan matrix following \cite{Ka90,Ku02}. We recall
definition of the Demazure lattice due to \cite{Le}.

Let $\Car:=(a_{ij})_{1\le i,j\le n}$ be a generalized Cartan matrix (i.e., $a_{ii}=2$, $-a_{ij}\in \Z_+$ for all $i\neq j$, where $\Z_+$
is the set of nonnegative integers, and $a_{ij}=0$ $\Leftrightarrow$ $a_{ji}=0$). Choose a triple $(\Upd,\sroot,\sdroot)$, unique up to isomorphism,
where $\Upd$ is a complex vector space of rank $2n-\rank A$, 
 \[\sroot=\{\al_1,\ldots,\al_n\}\subset \Up, \quad \sdroot=\{\al_1^\vee, \ldots,\al_n^\vee\}\subset \Upd\] are linearly independent subsets of simple roots and simple coroots respectively, satisfying  $a_{ij}=\pair{\al_j, \al_i^\vee}$. Such a triple is called the root datum corresponding to the matrix $A$.
 
Let $W$ be the Weyl group of the root datum, that is, the group generated by fundamental reflections \[s_i\colon \la\mapsto \la-\pair{\la, \al_i^\vee} \al_i,\quad \lambda\in \Up.\] It is a Coxeter group, and for  $i\neq j$, the order $m_{ij}$ of $s_is_j$ is equal to $2,3,4,6,\infty$ when $a_{ij}a_{ji}$ is $0,1,2,3,\ge 4$, respectively.
Let \[\ro=\{w (\al_i)\mid w\in W,\; 1\le i\le n\}\] denote the set of (real) roots. If $\al=w(\al_i)$, we define the reflection $s_\al(\la)=\la-\pair{\la,\al^\vee}\al$ where $\al^\vee=w(\al_i^\vee)$. We then have $s_\al=ws_iw^{-1}$.

Let $\proot=\ro\cap \sum_i\Z_+\al_i$ and $\nroot=-\proot$ denote the sets of positive and negative (real) roots respectively. We have
$\ro=\proot \amalg \nroot$.
For each $w\in W$, define 
\[
\len{w}=\{\al\in \proot\mid w^{-1}(\al)\in \nroot\}.
\]
The expression $w=s_{i_1}\ldots s_{i_l}\in W$ is called reduced if $l$ is minimal possible among all representations of $w\in W$ as a product
of the $s_i$ and in this case $l=\ell(w)$ is called the length of $w$. We then have
\begin{enumerate}
\item For each $w\in W$, $\al_i\in \sroot$, we have 
\[
\ell(s_iw)>\ell(w) \Leftrightarrow w^{-1}(\al_i)\in \proot\quad\text{ and }\quad\ell(s_iw)<\ell(w) \Leftrightarrow w^{-1}(\al_i)\in \nroot,
\]
\item $\len{s_i}=\{\al_i\}$, 
\item For a reduced expression $w=s_{i_1}\ldots s_{i_l}$ we have 
\[
\len{w}=\{\al_{i_1}, s_{i_1} \al_{i_{2}}, \ldots, s_{i_1}\cdots s_{i_{l-1}} \al_{i_l}\},\text{ in particular }\;\ell(w)=|\len{w}|.
\]
\end{enumerate}

Let $\rl=\sum_i\Z\al_i$ and $\rl^\vee=\sum_i\Z \al_i^\vee$ denote the root  and coroot lattices, respectively. Let $\cl$ be
a formal Demazure lattice of \cite[Def.~3.1]{Le}, i.e., a finitely generated free abelian subgroup of $\Upd^*$ such that $\pair{\cl,\rl^\vee} \subset \Z$ and
every simple root can be extended to a $\Z$-basis of $\cl$, i.e.,  
the coordinates of every simple root with respect to some $\Z$-basis of $\cl$ form a unimodular vector.

Observe that all these lattices are invariant under the action by $W$. 
We refer to \cite[\S3]{Le} for their basic properties  and examples.

\subsection{Formal Demazure operators}
We define formal Demazure algebra and Demazure operators
following \cite{CPZ, HMSZ, CZZ1, Le}.

Let $F$ be a one-dimensional commutative formal group law over a commutative ring $R$. We sometimes write $x+_Fy=F(x,y)$, and denote by $-_Fx\in R\lbr x\rbr$ the formal inverse of $x$.

Examples of formal group laws include the additive formal group law $F_a(x,y)=x+y$, the multiplicative formal group laws $ F_m(x,y)=x+y-\be xy$ for $\be \in R$, and the universal formal group law $F_u$ over the Lazard ring.

Let $\cl$ be the Demazure lattice and let $\RcF=R\lbr x_\la|\la\in {\cl}\rbr/J$, where $J$ is the closure of the ideal generated by $x_0$ and $x_{\la+\mu}-F(x_\la,x_\mu)$ with $\la,\mu\in \cl.$
$\RcF$ is called  the formal group algebra \cite{CPZ}. 
\begin{example}\label{ex:FaFm}
\begin{enumerate}
\item If $F=F_a$, then $\RcF\cong S_R^*(\cl)^\wedge$ is the completion of the symmetric algebra at the augmentation ideal. \item If $F=F_m$, then $\RcF\cong R[\cl]^{\wedge}$ is the completion of the group ring.
\end{enumerate}
\end{example}
Observe that the construction of $\RcF$ is functorial with respect to homomorphisms of rings $R\to R'$, homomorphisms of lattices $\cl \to \cl' $ and homomorphisms of formal group laws $F\to F'$. Moreover, the action of $W$ on $\cl$ induces an action of $W$ on $\RcF$ by $w(x_\la)=x_{w(\la)}$, $\la\in\cl$.

There is an augmentation map $\ep\colon\RcF\to R, x_\la\mapsto 0$, and let $\IF$ denote the kernel. We then have a filtration 
\[
\RcF=\IF^0\supset \IF\supset \IF^2\supset \cdots \supset \IF^i\supset \cdots,
\]
and define $\IF^{-i}=\RcF$ for $i>0$. The associated graded ring is denoted by 
\[
\Gr\RcF=\bigoplus_{i\ge 0}\IF^i/\IF^{i+1}.
\]
By \cite[Lemma 4.2]{CPZ}, we have $\Gr\RcF\cong S_R^*(\cl)$. For each   $q\in \IF^i\backslash \IF^{i+1}$, we say that $\deg q=i$ and denote by $(q)_i$ its image via the map
\[
\IF^i/\IF^{i+1}\hookrightarrow \Gr\RcF\cong S_R^*(\cl).
\]
We define $(q)_j=0$ for $j<\deg q$.

\begin{lem}\label{lem:irred} Assume that $R$ has characteristic 0. Let $\al\in \Phi_+$.
\begin{itemize}
\item[(1)]  $x_\al$ is irreducible in $\RcF$ for all positive root $\al$.
\item[(2)] If $x\in \RcF$ such that $x_\al|x_\be x$ for some positive roots $\al\neq \be$,  then $x_\al|x.$
\end{itemize}
\end{lem}
\begin{proof} (1) 
Applying $W$ we may assume that $\al=\al_i$ is some simple root. 
Let $\{\om_1,\ldots,\om_m\}$ be a basis of $\cl$ extending $\al_i$.
Then the coefficients $a_{ij}$ in $\al_i=\sum_{j=1}^na_{ij}\om_j$ are coprime.
Hence, the power series corresponding to $x_{\al_i}$ is irreducible in $R[\om_1,\ldots,\om_m]$.

(2) Again it suffices to assume that $\al$ is a simple root. Expressing $\be$ as a linear combination of simple roots, we get that $x_\be$ is regular in $S/(x_\al)$, so the conclusion follows.
\end{proof}

For each $v\in W$, define $c_v=\prod_{\al\in \len{v}}x_\al\in \RcF$. 
\begin{lem} \label{lem:cv} If $\ell(s_\al v)<\ell(v)$, then $x_\al\mid c_v$ and  $x_\al^2\nmid c_{v}$.
\end{lem}
\begin{proof}
If $\ell(s_\al v)<\ell(v)$, then $v^{-1}(\al)\in \nroot$ by \cite[Lemma 1.3.13]{Ku02}, hence $\al\in \len{v}$. So $x_\al\mid c_v.$ The second part follows since $x_\al\nmid x_\be$ for any $\al\neq \be\in \proot$.
\end{proof}

Following \cite[Corollary 3.4]{CPZ}
we define $R$-linear operators $\RcF \to \RcF$ as
\[
\Dem_\al(u)=\frac{u-s_\al(u)}{x_\al}, \quad\cC_\al(u)=\ka_\al u-\Dem_\al(u)=\frac{u}{x_{-\al}}+\frac{s_\al(u)}{x_\al},
\]
where $\kappa_\al:=\frac{1}{x_\al}+\frac{1}{x_{-\al}}\in \RcF$. The operator $\Dem_\al$ is called a formal Demazure operator, and $\cC_\al$ is called a formal push-pull operator.  They satisfy 
\[\Dem_\al\circ \Dem_\al=\ka_\al\Dem_\al=\Dem_\al \ka_\al,\quad  \cC_\al\circ \cC_\al=\ka_\al \cC_\al
\]
\[
\Dem_\al(uv)=\Dem_\al(u)v+s_\al(u)\Dem_\al(v),\quad u,v\in \RcF.
\]
Moreover, they are  $S^{W_\al}$-linear, where $W_\al$ is the subgroup of $W$ generated by $s_\al$. 

For simplicity, we set $x_{\pm i}=x_{\pm \al_i}$, $\Dem_{i}=\Dem_{\al_i}$ and $\cC_i=\cC_{\al_i}$. For a sequence $I=(i_1,\ldots,i_l)$ we set $\Dem_I=\Dem_{i_1}\circ \cdots \circ \Dem_{i_l}$.
If $E\subset [l]=\{1,\ldots,l\}$, we denote by $I|_E$ the sequence consisting of $i_j$ with $j\in E$.
According to \cite[Lemma 4.8]{CZZ1} we have 
\[
\Dem_I(uv)=\sum_{E_1,E_2\subseteq [l]}p_{E_1,E_2}^I\Dem_{I|_{E_1}}(u)\Dem_{I|_{E_2}}(v), \quad p^I_{E_1,E_2}\in \IF^{|E_1|+|E_2|-l}.
\]

\subsection{The formal affine Demazure algebra}
We now recall the definition of
the formal affine Demazure algebra associated to a generalized Cartan matrix following \cite[\S4]{Le}. 

Let $S=\RcF$ and let $Q=S[\frac{1}{x_\al}\mid\al\in \proot]$, be the localization of $S$ at all positive real roots. Define the twisted group algebra $Q_W=Q\rtimes R[W]$ to be the left $Q$-module  $Q\otimes_R R[W]$ with basis $\{\de_w\}_{w\in W}$, and define the product
\[
q\de_w\cdot q'\de_{w'}=qw(q')\de_{ww'}, \quad q,q'\in Q, \quad w,w'\in W.
\] 
We set $\de_\al=\de_{s_\al}$ and  $1=\de_e$. Define
\[
X_\al:=\frac{1}{x_\al}(1-\de_{\al})\quad\text{ and }\quad\oC_\al=\ka_\al-X_\al=\frac{1}{x_{-\al}}+\frac{1}{x_\al}\de_\al.
\]
The element $X_\al \in Q_W$ is called a formal Demazure element, and $Y_\al$ is called a formal push-pull element. We then have for any $q\in Q$, $\al\in \ro$, $w\in W$
\begin{gather*}
X_\al q=s_\al(q)X_\al+\Dem_\al (q), \\
X_\al^2=\kappa_\al X_\al=X_\al \kappa_\al, \quad \de_w X_\al \de_{w^{-1}}=X_{w(\al)},\\
Y_\al^2=\ka_\al Y_\al, \quad \de_w Y_\al \de_{w^{-1}}=Y_{w(\al)}. 
\end{gather*}

For simplicity, we set $\de_i=\de_{s_i}$, $X_i=X_{\al_i}$ and $Y_i=Y_{\al_i}$. For a sequence $I=(i_1,\ldots,i_l)$ we set $X_I=X_{i_1}\cdots X_{i_l}$; set $X_\emptyset =1$.  

Let $I_v$ denote a reduced sequence of $v\in W$, i.e., $I_v=(i_1,\ldots,i_l)$ for a reduced expression
$v=s_{i_1}\ldots s_{i_l}$. 
The proofs of the following two lemmas are the same as of \cite[Lemma 5.4]{CZZ1} (see also \cite[Lemma~3.2, 3.3]{CZZ2})
\begin{lem} \label{lem:deltatoX} We have $X_{I_v}=\sum_{w\le v}a^X_{v,w}\de_w$ for some $a^X_{v,w}\in Q$, $a^X_{v,v}=(-1)^{\ell(v)}\frac{1}{c_v}$, where the sum is taken over all $w\le v$ with respect to the  Bruhat order.

Moreover, we have $\de_v=\sum_{w\le v}b^X_{v,w}X_{I_w}$ for some $b^X_{v,w}\in S$, $b^X_{v,e}=1$ and $b^X_{v,v}=(-1)^{\ell(v)}c_v$. 
\end{lem}
\begin{lem} We have $Y_{I_v}=\sum_{w\le v}a^Y_{v,w}\de_w$ for some $a^Y_{v,w}\in Q$ and $a^Y_{v,v}=\frac{1}{c_v}$. 

Moreover, we have $\de_v=\sum_{w\le v}b^Y_{v,w}Y_{I_w}$ for some $b^Y_{v,w}\in S$ and $b^Y_{v,v}=c_v$. 
\end{lem}
\begin{cor} \label{cor:basisQW}The sets $\{X_{I_w}\}_{w\in W}$ and $\{Y_{I_w}\}_{w\in W}$ form bases of $Q_W$ as a left $Q$-module.
\end{cor}

\begin{defi}Let $\DcF$ be the $R$-subalgebra of $Q_W$ generated by elements of $S$ and the Demazure elements $\{X_{\al}\}_{\al\in \ro}$. We call $\DcF$ the formal affine Demazure algebra associated to the generalized Cartan matrix $A$ and the Demazure lattice $\cl$.
\end{defi}

Observe that $\DcF$ is generated by elements of $S$ and $X_i$, $i\in [n]$. Indeed, we have $\de_i=1-x_i X_i$, so $\de_w\in \DcF$. Since any root $\al$ can be written as $w(\al_i)$ for some $w\in W$ and a simple root $\al_i$, and  $X_{w(\al_i)}=\de_w X_i\de_{w^{-1}}$, the conclusion then follows.

By definition the algebras $Q_W$ and $\DcF$ are functorial with respect to morphisms of rings $R\to R'$, morphisms of formal group laws $F\to F'$, and morphisms of root data, i.e., morphisms of lattices $\phi\colon \cl\to \cl'$ sending roots (resp. coroots) to roots (resp. coroots), and such that $\phi^\vee(\phi(\al)^\vee)=\al^\vee$.

\section{General position arguments and divisibility by positive roots}\label{sec:genpos}
The purpose of the present section is to extend \cite[Lemma 4.9]{KK86} (additive case $F=F_a$) and \cite[Lemma 2.12]{KK90} (multiplicative case $F=F_m$) to an arbitrary formal group law $F$. This is done in Proposition~\ref{lem:keyalpha} which will be used in the proof of  our main theorem (Theorem~\ref{thm:mainDem}). 
Observe that in the additive (resp. multiplicative) case the ring $S$ is the ring of (resp. Laurent) polynomials (non-complete version of Example~\ref{ex:FaFm}), so it can be viewed as a ring of functions on the lattice $\Lambda$. The respective proofs by Kostant-Kumar then deal with analyzing zeros and poles of such functions. For a general $F$, the ring $S$ becomes the ring of formal power series, so we no longer have functions on $\Lambda$ as well as zeros and poles. To overcome this difficulty, we use the language of prime ideals instead.

Suppose $R$ is a domain of characteristic $0$. Let $K$ denote its field of fractions.
Recall that $S=\RcF$. We set
\[S_{K}=K[[\cl]]_F, ~S_{a,K}=K[[\cl]]_{F_a},~ S'=S_K^*(\cl),\]
\[Z=\Spec(S), ~Z_K=\Spec(S_{K}),~ Z_{a,K}=\Spec(S_{a,K}),~Y=\Spec(S')\cong \mathbb{A}^n_K.\]
There are induced actions of $W$ on $S$, $S_K$, $S_{a,K}$ and $S'$, hence, on  $Z$, $Z_K$, $Z_{a,K}$ and $Y$. 
An element $x_\la$ in $S_{a,K}$ (or in $S'$) will be simply denoted by $\lambda$. 

There exists an isomorphism  $h\colon F_{a}\to F$ of formal group laws over $K$, i.e., $h(x)\in K[[x]]$ is such that 
$h(x)+_Fh(y)=h(x+y)$. It induces
an isomorphism
\[
H\colon S_K\to S_{a, K}, ~x_{\al_i}\mapsto h(\al_i). 
\]

\begin{lem}The following diagram commutes
\[
\xymatrix{Z  \ar[d]_{w\cdot} & Z_K\ar[l]_\phi\ar[d]^{w\cdot} & Z_{a,K}\ar[l]_H\ar[r]^\psi\ar[d]^{w\cdot} & Y\ar[d]^{w\cdot}\\
    Z   & Z_K \ar[l]_\phi& Z_{a,K}\ar[l]_H\ar[r]^\psi & Y}
\]
where $w\cdot$ is the action of $w\in W$ on corresponding spectrum, $\phi$ is the map induced by the embedding $R\hookrightarrow K$, and $\psi$ is the map induced by the embedding $S'\hookrightarrow S_{a,K}$ (observe that $S_{a,K}$ is the completion of $S'$ along the augmentation ideal).
\end{lem}
\begin{proof}The commutativity of the left and right squares are obvious. Concerning the middle square, we have
\[
w(H(x_\la))=w(h(\la))=h(w(\la))=H(x_{w(\la)})=H(w(x_\la)). \qedhere
\]
\end{proof}

We define the $w$-invariant subset $Z^w:=\{x\in Z\mid w(x)=x\}$, and similarly define $Z_{K}^w$, $Z_{a,K}^w$, $Y^w$. For a commutative ring $R'$ and an element $r\in R'$, we denote by $V(r)$ the set of prime ideals of $R'$ containing $r$. Then we have
\[\phi^{-1}(V(x_\al))= V(x_\al),~ H(V(\al))=H(V(h(\al)))=V(x_\al),~ \psi^{-1}(V(\al))=V(\al),\]
\[
\phi^{-1}(Z^w)=Z_K^w, ~H(Z^w_{a,K})=Z^w_{K},~ \psi^{-1}(Y^w)=Z^w_{a,K}.
\]
The identity
$V(h(\al))=V(\al)$ follows from the fact that $h(x)$ can be written as $h(x)=xg(x)$ for some invertible $g(x)\in K\lbr x\rbr$.

We will use the following general position-type result

\begin{lem} \label{thm:dense} Suppose $R$ is a domain of characteristic 0. For any root $\al\in \ro$, the set 
\[
P_\al:=V(x_\al)\backslash\big([\cup_{\be\in \Phi,\; \be\neq \al}V(x_\be)]\cup[\cup_{w\neq e,s_\al}Z^w]\big)
\]
is a dense subset of $V(x_\al)\subset Z$.
\end{lem}
\begin{proof} 
It suffices to prove that
$\overline{\phi^{-1}(P_\al)}=V(x_\al)\subset Z_K$. Indeed, since  $\phi$ is dominant,  which means it maps dense subsets to dense subsets, $\overline{\phi(\phi^{-1}(P_\al))}=\overline{\phi(V(x_\al))}=V(x_\al)\subset Z$, and we know $\phi(\phi^{-1}(P_\al))\subset P_\al$, so $\overline{P_\al}=V(x_\al)$.

 The map $H$ is an isomorphism, and  $\psi$ is flat, so the inverse image via $\psi$ of a dense subset is dense (\cite[Corollary I.2.8]{Mi80}). 
Therefore,  it is enough to show that 
\[
P'_\al:=V(\al)\backslash \big([\cup_{\be\neq \al}V(\be)]\cup[\cup_{w\neq e, s_\al}Y^w]\big)
\] 
is dense in $V(\al)\subset Y$. We can assume that $K=\Q$.  
By functoriality of formal group algebras we may assume that $\cl=\rl$ and $\al=\al_1$ is a simple root. Then $V(\be)$ is the hyperplane in $Y=\mathbb{A}^n_\Q$ orthogonal to $\be$.

Consider the positive half-hyperplane
\[
V(\al_1)^+:=\{\bar a=(0,a_2,\ldots,a_n)\in Y\mid a_i> 0 \text{ for }i\ge 2\} \subset V(\al_1).
\]
We claim that $P'_{\al_1}$ contains $V(\al_1)^+$ which would immediately imply that $P'_{\al_1}$ is dense in $V(\al_1)$.

Indeed, since $\be=c_1\al_1+\ldots+c_n\al_n$ where $c_i$'s are either all positive or negative, 
$\be(\bar a)=c_2a_2+\ldots +c_n a_n\neq 0$ unless $c_2=\ldots =c_n=0$ (that is $\be=\al_1$). Hence, 
$V(\al_1)^+\cap V(\be)=\emptyset$ for all $\be\neq \al_1$.  Finally, by~\cite[Proposition~3.12.(a)]{Ka90} the stabilizer $W_{\bar a}:=\{w\in W\mid w(\bar a)=\bar a\}$ is generated by  simple reflections it contains, so $W_{\bar a}=\{e,s_1\}$. Hence, $V(\al_1)^+\cap Y^w=\emptyset$ for all $w\neq e,s_1$.
\end{proof}

Consider an action of the twisted group algebra $Q_W$ on $Q$ via
\[
q\de_w\cdot q'=qw(q'), \quad q,q'\in Q, w\in W.
\]
As before, let $X_{I_w}$ denote the product of Demazure elements indexed by a reduced expression for $w\in W$.
Recall that (see Corollary~\ref{cor:basisQW}) $X_{I_w}$, $w\in W$ form a basis of $Q_W$ as a left $Q$-module.

We obtain the following generalization of \cite[Lemma~4.10]{KK86}

\begin{prop}\label{lem:keyalpha} 
Let $z=\sum_{w\in W,\; \ell(w)\le k} p_w X_{I_w}$, $p_w\in S$ and let $\al\in \proot$. 

If $z\cdot S\subseteq x_\al S$, then $x_\al \mid p_w$ for all $w$.
\end{prop}

\begin{proof}
We may assume $p_w\neq 0$ for some $w$ with $\ell(w)=k$.
Set $z^+=z+\de_\al z$ and $z^-=z-\de_\al z$ so that $2z=z^++z^-$. 
Since $\frac{x_\al}{x_{-\al}}$ is invertible in $S$, we have
\[
(\de_\al z) \cdot S\subseteq \de_\al\cdot (x_\al S)=x_{-\al}S=x_\al S.
\]
By Lemma \ref{lem:irred} $x_\al$ is irreducible, so $x_\al\mid 2f$ $\Rightarrow$ $x_\al\mid f$ for any $f\in S$. 
Hence, it suffices to prove the lemma for $z^+$ and $z^-$, i.e., we may assume $\de_\al z=\pm z$. 
Moreover, it is enough to show that $x_\al\mid 2p_w$ for all $w$.

Expressing $z$ in terms of the canonical basis $\{\delta_w\}$ we obtain
\[
z=\sum_{\ell(w)\le k}p_wX_{I_w}=\sum_{\ell(w)\le k}q_w\de_w, ~q_w\in Q.\]
Choose an element $w_0$ of length $k$ such that $p_{w_0}\neq 0$.
By Lemma \ref{lem:deltatoX}  $q_{w_0}=\tfrac{1}{c_{w_0}}p_{w_0}$. Since $\de_\al z=\pm z$,  we have $\ell(s_\al w_0)<\ell(w_0)$, hence by Lemma \ref{lem:cv}, $c_{w_0}=x_\al c'$ where $c'\in S$ is a product of some $x_\be$ with $\be\in \proot \backslash\{\al\}$.  On the other hand, $\de_\al z=\pm z$ implies  that  $q_{s_\al w_0}=\pm s_\al(q_{w_0})$. Hence, we can write 
\[
z=\tfrac{p_{w_0}}{x_\al c'}\de_{w_0}\pm\tfrac{s_\al ( p_{w_0})}{x_{-\al }s_\al(c')}\de_{s_\al w_0}+z',\text{ where }
z':=\sum_{\ell(w)\le k,\; w\neq w_0, s_\al w_0}q_w \de_w.\]

We claim that $x_\al\mid 2p_{w_0}$. Indeed, it suffices to prove that for any $\frakp\in V(x_\al )\subsetneq Z=\Spec S$ we have $2p_{w_0}\in \frakp$.  By Lemma~\ref{thm:dense}, the set \[
P_\al=V(x_\al)\backslash \big( [\cup_{\be\neq \al}V(x_\be)]\cup [\cup_{w\neq e,s_\al}Z^w]\big)
\]
is dense in $V(x_\al)$, so it suffices to prove that  for any $\frakp\in P_\al$ we have $2p_{w_0}\in \frakp$.

Let $\frakp\in P_\al$. 
By definition, $w(\frakp)\neq \frakp$ for any $w\neq e,s_\al$.
Hence, for any $w\neq w_0,s_\al w_0$ we have $w_0^{-1}(\frakp)\neq w^{-1}(\frakp)$ and $w_0^{-1}s_\al(\frakp)\neq w^{-1}(\frakp)$ which implies that $w^{-1}(\frakp)\backslash \big(w_0^{-1}(\frakp)\cup w_0^{-1}s_\al (\frakp)\big)\neq \emptyset$. The latter means that for any $w\neq w_0,s_\al w_0$ there exists $r_w\in S$ such that
$w(r_w)\in \frakp$, $w_0(r_w)\not\in \frakp$ and $s_\al w_0(r_w)\not\in \frakp$.
We define
\[
r=\prod_{q_w\neq 0 \text{ and } w\neq w_0, s_\al w_0}r_w^{n_w+1},
\] where  $n_w$ is the order of $x_\al $ in the denominator of $q_w$. 

Let $\tilde{S}$ be the localization of $S$ at all $x_\be$ with $\be\in \proot \backslash\{\al\}$, then $Q=\tilde{S}[\frac{1}{x_\al }]$.
Observe that $\frac{1}{c'},\frac{1}{ s_\al(c')}\in \tilde S$.
Since  $x_\be\not\in \frakp$ for any such $\be$, $\frakp_{\tilde S}:=\frakp\tilde{S}$ is a proper prime ideal of $\tilde{S}$. 
The assumption $z\cdot S\subseteq x_\al S$ implies that  $z\cdot r\in (x_\al) \subset \frakp_{\tilde{S}}$.

(i) Suppose that $\de_\al z=-z$. We obtain $x_\al z\cdot r=x_\al(z\cdot r)\in \frakp_{\tilde{S}}$, and 
\[
\eta:=[\tfrac{p_{w_0}}{c'}\de_{w_0}-\tfrac{ x_\al s_\al (p_{w_0})}{x_{-\al }s_\al(c')}\de_{s_\al w_0}]\cdot r=\tfrac{p_{w_0}}{c'}w_0(r)-\tfrac{x_\al s_\al (p_{w_0})}{x_{-\al }s_\al(c')}s_\al w_0(r)\in \tilde{S}.
\]
Note that $\eta=x_\al z\cdot r-x_\al z'\cdot r$, so $x_\al z'\cdot r\in \tilde{S}$. By Lemma~\ref{lem:zprime}, we have $x_\al z'\cdot r\in \frakp_{\tilde S}$, which  implies that  $\eta\in \frakp_{\tilde{S}}$.

 Denote $\zeta:=\frac{p_{w_0}w_0(r)}{c'}$, then $\eta=\zeta-\frac{x_\al }{x_{-\al }}s_\al (\zeta)\in \frakp_{\tilde{S}}$. Note that $\frac{x_\al }{x_{-\al }}+1$ is divisible by $x_\al \in \frakp$, and $s_\al (\zeta)-\zeta=-x_\al \Dem_\al (\zeta)\in \frakp$, so $2\zeta\in \frakp_{\tilde{S}}$. Since $w_0(r)\not\in \frakp$, we get  $2p_{w_0}\in \frakp$. 

(ii) Suppose that $\de_\al z=z$. Similarly, we obtain $z\cdot (w_0^{-1}(x_\al )r)\in \frakp$. Moreover, 
\[
\eta':=[\tfrac{p_{w_0}}{x_\al c'}\de_{w_0}+\tfrac{s_\al (p_{w_0})}{x_{-\al }s_\al(c')}\de_{s_\al w_0}]\cdot (w_0^{-1}(x_\al )r)=\tfrac{p_{w_0}}{c'}w_0(r)+\tfrac{s_\al (p_{w_0})}{s_\al(c')}s_\al w_0(r) \in  \tilde{S}.
\]
Therefore, $z'\cdot (w_0^{-1}(x_\al )r)=z\cdot w_0^{-1}(x_\al)r -\eta'\in \tilde{S}$. Similar to Lemma~\ref{lem:zprime}, we have $ z'\cdot (w_0^{-1}(x_\al) r)\in \frakp_{\tilde{S}}$. Therefore, $\eta'\in \frakp_{\tilde{S}}$ $\Rightarrow$ $\frac{2p_{w_0}}{c'}w_0(r)\in \frakp_{\tilde{S}}$ $\Rightarrow$ $2p_{w_0}\in \frakp$. 
\end{proof}

\begin{lem}\label{lem:zprime} With the notation of the proof of Lemma \ref{lem:keyalpha}, we have $x_\al z'\cdot r\in \frakp_{\tilde S}$.
\end{lem}
\begin{proof} By definition, for $w\neq w_0, s_\al w_0$ we have
\[q_w\de_w\cdot r=q_w\prod_{v\neq w_0, s_\al w_0}w(r_v^{n_v+1}).\]
If $x_\al$ does not divide the denominator of $q_w$, then $q_w\in \tilde{S}$, so $q_w\de_w\cdot r\in\frakp_{\tilde{S}}$ since $w(r_w)\in \frakp$. If $x_\al$ divides the denominator of $q_w$, then by definition of $r$ the numerator of $q_w\de_w\cdot r$ belongs to $\frakp\subset \frakp_{\tilde S}$. 

Since $Q=\tilde{S}[\frac{1}{x_\al }]$,  we can rewrite $x_\al z'\cdot r\in \tilde S$ as 
\[
x_\al z'\cdot r=\sum_{j=0}^m\tfrac{q_j'}{x_\al ^j}\in \tilde{S}, \quad q'_j\in \frakp_{\tilde{S}} \text{ for all }j, \text{ and } x_\al \nmid q_j' \text{ for }j>0.
\]
If $m>0$, then reducing the expression to the common denominator $x_\al ^m$, we obtain that $x_\al \mid {\sum_{j=0}^mq'_jx_\al ^{m-j}}$, hence $x_\al \mid q'_m$, a contradiction. So $m=0$ and the conclusion follows. 
\end{proof}

\section{The formal affine Demazure algebra and its dual}
In this section we prove the structure theorem (Theorem \ref{thm:mainDem}) 
which the key result of this paper. It allows to extend most of the properties and facts concerning the algebra $\DcF$ and its dual $\DcFd$ from the finite case
to the case of an arbitary generalized Cartan matrix.

As in the previous section, we  consider the $Q_W$ action on $Q$.
It induces an action of the formal affine Demazure algebra $\DcF$ on $S=\RcF$ via
\[
X_i\cdot q=\tfrac{1}{x_i}(1-\de_i)\cdot q=\Dem_i(q), ~q\in S, 
\]
so we have $\DcF\cdot S\subseteq S$.

The following theorem generalizes
\cite[Theorem~4.6]{KK86}, \cite[Theorem~2.9]{KK90}, and \cite[Proposition~4.13]{Le} 
to the context of an arbitrary formal group law and an arbitrary Demazure lattice.

\begin{theo}\label{thm:mainDem} Let $R$ be a domain of characteristic 0. Let $\DcF$ be the formal affine Demazure algebra
defined for a given generalized Cartan matrix, Demazure lattice and a formal group law over $R$. Then 
\[
\DcF=\{z\in \QW\mid z\cdot S\subseteq S\}.
\]
\end{theo}

\begin{proof}
By definition we have $\DcF\subseteq\{z\in Q_W\mid z\cdot S\subseteq S\}$. To prove the opposite inclusion let $z\in Q_W$ be such that $z\cdot S\subseteq S$. Then by Lemma \ref{lem:deltatoX}, $\{X_{I_w}\}_{w\in W}$ is a basis of $Q_W$, so we can write $z=\frac{1}{p}\sum_{\ell(w)\le k}p_wX_{I_w}$ with $p_w\in S$ and $p$ is a product of (possibly repeated) $x_\al, \al\in \Phi_+$. It suffices to assume that $p$ is irreducible, and by Lemma \ref{lem:irred}, we can assume that $p=x_\al$.
Then Proposition~\ref{lem:keyalpha} implies that $x_\al\mid p_w$  for all $w$. Therefore, $z\in \DcF.$
\end{proof}

As an immediate consequence we obtain
\begin{cor}The set $\{X_{I_w}\}_{w\in W}$ forms a basis of $\DcF$ as a left (or right) $S$-module.
\end{cor}
\begin{proof}
The set $\{X_{I_w}\}_{w\in W}$ spans $\DcF$ as a left $S$-module. It follows from Corollary~\ref{cor:basisQW} that this set is linearly independent. 
\end{proof}

Based on Theorem~\ref{thm:mainDem}, the following corollaries hold, whose proofs are exactly the same as in the finite case.

First, we obtain a residue construction of the algebra $\DcF$ for an arbitrary generalized Cartan matrix and a Demazure lattice. Such construction was first mentioned in \cite{GKV97} for elliptic curves and then proved in the finite case in  \cite[\S4]{ZZ}.
\begin{cor} We have
\[
\DcF=\{\sum_{w\in W}a_w\de_w \in Q_W\mid x_\al a_w\in Q_\al\text{ and }a_w+a_{s_\al w}\in Q_\al,\; \forall \al\in \proot\}
\]
where $Q_\al=S[\frac{1}{x_\be}\mid \be\in \proot, \be\neq \al]$ denotes the respective localization of $S$.
\end{cor}

Next, we obtain a uniform description of all relations between generators in $\DcF$.
Observe that in the finite case explicit relations were given in \cite[Proposition 6.8]{HMSZ} and \cite[Lemma 7.1]{CZZ1};
for a hyperbolic formal group law and a generalized Cartan matrix it was given in \cite[Example~4.12]{Le} based on the explicit presentation of \cite{HMSZ}.

\begin{cor}\label{cor:present} If $I_w$ and $I_w'$ are two reduced expressions of $w$, then 
\[
X_{I_w}-X_{I_{w}'}=\sum_{v<w,\; \ell(v)\le \ell(w)-2}p_{w,I_v}X_{I_v}, ~p_{w,I_v}\in S.
\]
\end{cor}
In particular, if $(s_is_j)^{m_{ij}}=1$ and $w=\underbrace{s_is_js_i\cdots }_{m_{ij} \text{ times }}$, then
\[
\underbrace{X_{i}X_jX_i\cdots}_{m_{ij} \text{ times}}-\underbrace{X_jX_iX_j\cdots}_{m_{ij} \text{ times}}=\sum_{v<w,\; \ell(v)\le \ell(w)-2}p_vX_{I_v}, ~p_v\in S.
\]
Observe that if $F=F_a$ or $F_m$, then all $p_v$ vanish (\cite[Proposition~4.2]{KK86} and \cite[Proposition~2.4]{KK90}) and we obtain the usual braid relation.

We then identify $\DcF$ with a subring of $R$-linear operators on $S$, that is

\begin{cor} \label{cor:injDem}The map $\DcF\to \End_R(S)$, $X_i\mapsto \Delta_i$ is injective.
\end{cor}
\begin{proof}Let $z\in \DcF$ be in the kernel. Write $z=\sum_{\ell(w)\le k}p_wX_{I_w}$ with $p_w\in S$. If $z\cdot S=0$, then  $z\cdot S\subseteq x_\al S$. So by Lemma~\ref{lem:keyalpha}, $x_\al\mid p_w$ for any $\al$ and any $w$. Then $z':=\frac{1}{x_\al}z$ satisfies the same condition as $z$ does, so $x_\al\mid \frac{p_w}{x_\al}$. Recursively, we see that $p_w=0$ for all $w$.
\end{proof}

\begin{example}\label{prop:addcompareDem} Consider the additive case, i.e. $F=F_a$.
In this case,  $X_{I_w}$ does not depend on the choice of the reduced decomposition $I_w$ of $w$, so we denote it simply by $X_w$. Let $Q_{a}=S^*_R(\cl)[\frac{1}{\al}|\al\in \Phi_+]$ and $Q_{a,W}=Q_a\rtimes R[W]$ with basis $\{\de_w^a\}_{w\in W}$. The map $S^*_R(\cl)\to R\lbr\cl\rbr_{F_a}$ induces a map between twisted group algebras and, hence, an embedding \[\nila\hookrightarrow \DcFa\text{ with  }x_w\mapsto (-1)^{\ell(w)}X_w,\] where $\nila$ is the nil (affine) Hecke algebra with basis $\{x_w\}_{w\in W}$ of \cite[\S4]{KK86}
and $\DcFa$ denotes $\mathbf{D}_{F_a}$. Moreover, we have
\[R\lbr \cl \rbr_{F_a}\otimes_{S^*_R(\cl)}\nila\simeq\DcFa\quad \text{ as }R\lbr \cl \rbr_{F_a}\text{-modules.}\]
Similar to Lemma \ref{lem:deltatoX}, we obtain 
\[
x_w=\sum_{v\le w}a_{w,v}\de_v, ~ \de_w=\sum_{v\le e}b_{w,v}x_v,~\text{with }  b_{w,v}\in S^*_R(\cl), ~ a_{w,v}\in Q_a.
\]
The coefficient $a_{w,v}=(-1)^{\ell(w)}a_{w,v}^X$ corresponds to $c_{w^{-1}, v^{-1}}$ in \cite[(4.3)]{KK86}.
\end{example}

We now turn to the study of the dual of the algebra $\DcF$. As in \cite{CZZ1} we, first,
introduce a coproduct structure on $Q_W$. We view $Q_W\otimes_QQ_W$ as the tensor product of left $Q$-modules, and define
\[
\Tr: Q_W\to Q_W\otimes_QQ_W, ~q\de_w\mapsto q\de_w\otimes \de_w.
\]
The counit is $\ep\colon Q_W\to Q$, $q\de_w\mapsto q$. We define a product structure on $Q_W\otimes_QQ_W$ by 
\[
(z_1\otimes z_2)\cdot (z_1'\otimes z_2')=z_1z_1'\otimes z_2z_2'.
\]
By \cite[Proposition 8.10]{CZZ1} $\Tr$ is a ring homomorphism, so $Q_W$ is a cocommutative coalgebra in the category of left $Q$-modules.
\begin{rem}\ We can also view the counit as $\ep\colon Q_W\to Q$, $z\mapsto z\cdot 1$ where $\cdot$ is the action of $Q_W$ on $Q$. In particular, it implies that $\ep(zq)=z\cdot q$ for $z\in Q_W$ and $q\in Q$. 
\end{rem}

Let $Q_W^*$ be the dual of $Q_W$, i.e., $Q_W^*=\Hom_Q(Q_W,Q)$. By definition, it is a $Q$-module by 
\[
(q f)(z)=qf(z), \quad q\in Q, z\in Q_W, f\in Q_W^*.
\]
Indeed, $Q_W^*$ is equal to the left $Q$-module $\Hom(W,Q)$, the set of maps (of sets) from $W$ to $Q$. Define $\{f_v\}_{v\in W}\subset Q_W^*$ by $f_w(\de_v)=\de_{v,w}$ (the Kronecker symbol), then each $f\in Q_W^*$ can be written as $f=\sum'_{w\in W}q_wf_w$ with $q_w=f(\de_w)$, where $\sum'$ denotes a possibly infinite (formal) sum. Since $\{X_{I_w}\}_{w\in W}$ is another basis of $Q_W$, we could define $\{X_{I_w}^*\}_{w\in W}$ such that $X_{I_w}^*(X_{I_v})=\de_{w,v}$. Then any $f\in Q_W^*$ can be  written as $f=\sum'_{w\in W}f(X_{I_w})X_{I_w}^*$.
\begin{lem} \cite[Proposition 9.5]{CZZ1} For any sequence $I=(i_1,...,i_k)$, we have 
\[
\TR(X_I)=\sum_{E_1,E_2\subseteq [k]}p^I_{E_1,E_2}X_{I_{E_1}}\otimes X_{I_{E_2}}, \quad \text{where } p_{E_1,E_2}^I \in Q.
\]
\end{lem}

\begin{lem}\cite[Theorem 9.2]{CZZ1} The coproduct on $Q_W$ induces a coproduct structure on $\DcF$, which makes it into a cocommutative coalgebra with counit $\ep\colon \DcF\to S$ induced by $\ep\colon Q_W\to Q$. Moreover, $\ep(X_{I_w})=\de_{e,w}$.
\end{lem}
Consider the dual $\DcFd=\Hom_S(\DcF, S)$ where $\DcF$ is considered as a left $S$-module. Then the coproduct structure on $\DcF$ makes $\DcFd$ into a commutative ring with unit $\unit=\ep\in \DcFd$. The set of elements $\{X^*_{I_w}\}_{w\in W}$ can also be viewed as in $\DcFd$, and each $f\in \DcFd$ can be uniquely written as $f=\sum'_{w\in W}q_wX^*_{I_w}$ with $q_w\in S$. Moreover, as in \cite[Lemma 10.2]{CZZ2}, we have 
\begin{lem}The algebra $\DcFd$ can be identified with the subset 
\[
\{f\in Q_W^*\mid f(\DcF)\subseteq S\}.
\]
In particular, $\unit=X^*_{I_e}\in \DcFd$.
\end{lem}
\begin{proof} The first part follows as in the proof of \cite[Lemma 9.2]{CZZ2}. To prove the second part, we have  $\ep(X_{I_w})=X_{I_w}\cdot 1=\de_{w,e}\in S$, so $\unit(X_{I_w}^*)=\de_{w,e}$. That is, $\unit=X^*_{I_e}$.
\end{proof}

\begin{rem}In the case $F=F_a$, if $f$, $g\in Q_W^*$ are such that $f(X_{I_w})=0$ and $g(X_{I_w})=0$ for all but finitely many $w\in W$, then $fg(X_{I_w})=0$ for all but finitely many $w\in W$. This implies that the dual of the nil Hecke algebra $\nila$ (see Example~\ref{prop:addcompareDem}) has a product structure.

For a general $F$ it is unknown whether such finiteness condition holds. Consequently, if one restricts to the $S$-submodule of $\DcFd$ consisting of $f$ such that $f(X_{I_w})=0$ for all but finite many of $w\in W$, then this submodule maybe not be a ring. This is one of the reasons why the definition of $\Psi$  in \cite[Definition 2.19]{KK90}  (the case $F=F_m$) does not include the finiteness condition.
\end{rem}

\section{Filtrations on the Demazure algebra and its dual}
In this section we introduce and study filtrations on $\DcF$ and its dual $\DcFd$.
The filtration on $\DcFd$ come from the filtration by the codimension of Schubert basis
and generalizes the one for $\nila^*$ of \cite[Definition 2.28]{KK90}. 
Our main result (Theorem~\ref{thm:grnil}) generalizes \cite[Proposition 2.30]{KK90}.

Set $\DcF^{(i)}$ to be the $R$-submodule of $\DcF$ generated by $qX_{I}$ where $q\in S$ and $\deg q-|I|\ge i$. It gives a filtration \[
\cdots  \DcF^{(-i)}\supsetneq \cdots \supsetneq \DcF^{(0)}\supsetneq \cdots \supsetneq \DcF^{(i)}\supsetneq \cdots.
\]
As in \cite[Proposition 3.6]{CPZ} we conclude that $\DcF^{(i)}\cdot \IF^j\subset \IF^{i+j}$. Moreover,  since $\{X_{I_w}\}_{w\in W}$ is a basis of $\DcF$,  the submodule $\DcF^{(i)}$ is spanned by $qX_{I_w}$ with $q\in S$ and $\deg q-\ell(w)\ge i$. 
\begin{lem}\label{lem:filprod}
\begin{enumerate}
\item For any $q\in S$ and a sequence $I$, we have \[X_Iq=\sum_{E\subseteq I}\phi_{I,E}(q)X_E,\] where $\phi_{I,E}(q)\in S$ has degree $\deg q-|I|+|E|$.
\item For any $i$, $j$  we have $\DcF^{(i)}\cdot \DcF^{(j)}\subseteq \DcF^{(i+j)}$.
\end{enumerate}
\end{lem}
\begin{proof}
(a) The identity follows from \cite[Lemma 9.3]{CZZ1}. If $I=(i)$, then \[X_iq=s_i(q)X_i+\Dem_i(q)\text{ with }\deg s_i(q)=\deg q\text{ and }\deg \Dem_i(q)=\deg q-1.\] The formula for the degree then follows by induction on $|I|$.

(b) It suffices to show that $pX_IqX_J\in \DcF^{(i+j)}$ if $\deg p-|I|\ge i$ and $\deg q-|J|\ge j$. By part (a) we have \[X_Iq=\sum_{E\subseteq I}\phi_{I,E}(q)X_{E},\text{ where }\deg \phi_{I,E}(q)=\deg q-|I|+|E|.\] So $pX_IqX_J=\sum_{E\subseteq I}p\phi_{I,E}(q)X_{E}X_J$ where 
\[
\deg p+\deg \phi_{I,E}(q)-|E|-|J|=\deg p+\deg q-|I|-|J|\ge i+j.
\]
Hence, $pX_IqX_J\in \DcF^{(i+j)}$.
\end{proof}

Let $\Gr \DcF=\bigoplus_{i\in \Z}\DcF^{(i/i+1)}$, where $\DcF^{(i/i+1)}=\DcF^{(i)}/\DcF^{(i+1)}$, be the associated graded module. By Lemma~\ref{lem:filprod} it respects the product structure and, hence, defines a ring. Consider a map 
\begin{equation}\label{eq:FtoFA}
\eta_i\colon \DcF^{(i)}\to \nila, \quad qX_{I_w} \mapsto (-1)^{\ell(w)} (q)_{i+\ell(w)} x_{w},~q\in S,
\end{equation}
where $\nila$ is the nil affine Hecke algebra of \cite{KK86} (cf. Example \ref{prop:addcompareDem}).
By definition, $\eta_i(qX_{I_w})=0$ if $\deg q-\ell(w)=i+1$.
So it factors through $\eta_i\colon \DcF^{(i/i+1)}\to \nila$. 
\begin{prop}(cf. \cite[Proposition 4.8]{CPZ})\label{prop:filtration} The map $\oplus \eta_i\colon \Gr \DcF\to \nila$ is an isomorphism of $\Gr \RcF\cong S_R^*(\cl)$-modules, and it is also an isomorphism of rings.
\end{prop}
\begin{proof}It follows from Corollary \ref{cor:injDem} that $\DcF$ is isomorphic to the subalgebra of $\End_R(S)$ generated by $S$ and by $\Dem_i$, $1\le i\le l$. As in the proof of~\cite[Proposition~4.8]{CPZ}, we have an isomorphism $\Gr\DcF\simeq \nila$ of  $\Gr\RcF\simeq S_R^*(\cl)$-modules.
Similar to the \cite[Proposition 4.4]{CPZ}, the action on $\Gr \RcF$ of the class of $qX_{I_w}$ in $\Gr\DcF$  is the same as the action of  $(-1)^{\ell(w)}(q)_{i+\ell(w)}x_w$ on $S_R^*(\cl)$. In particular, it shows that $\eta_i$ preserves the product. Therefore, $\Gr \DcF\simeq \nila$ as rings.
\end{proof}

We define a filtration on $\DcFd$ as follows: for $i\ge 0$ we set 
\[(\DcFd)^{(i)}=\{f\in \DcFd\mid f(\de_w)\in \IF^i \text{ for all }w\in W\}.\]
Then 
\[
\DcFd=(\DcFd)^{(0)}\supsetneq\cdots \supsetneq (\DcFd)^{(i)}\supsetneq (\DcFd)^{(i+1)}\supsetneq \cdots,
\]
and since  $\TR( \de_w)= \de_w\otimes \de_w$, we have 
\[
(\DcFd)^{(i)}(\DcFd)^{(i')}\subseteq(\DcFd)^{(i+i')}.
\]
Let $\Gr\DcFd=\bigoplus_{i\ge 0} (\DcFd)^{(i/i+1)}$, where $(\DcFd)^{(i/i+1)}=(\DcFd)^{(i)}/(\DcFd)^{(i+1)}$,  be the associated graded ring. 
\begin{lem} \label{lem:dualfil}(a) For any $v=s_{i_1}\ldots s_{i_k}$, express \[\de_v=(1-x_{i_1}X_{i_1})\ldots (1-x_{i_k}X_{i_k})=\sum_{I\subset \{i_1,...,i_k\}}p_IX_I.\] 

Then $p_I\in S$ and  $\deg p_I\ge |I|$.

(b) For any $v$, express $\de_v=\sum_{w\le v}b^X_{v,w}X_{I_w}$ as in Lemma~\ref{lem:deltatoX}. 

Then $\deg b^X_{v,w}\ge \ell(w)$. Moreover, $(b^X_{v,w})_{\ell(w)}=(-1)^{\ell(w)}b_{v,w}\in S^*_R(\cl)$ where $b_{v,w}$ was defined in Example \ref{prop:addcompareDem}.
\end{lem}
\begin{proof}(a) If $k=1$, then $\de_{i_1}=1-x_{i_1}X_{i_1}$, so the conclusion holds. Now suppose it holds for $k$. Then $\de_{v}=\sum p_IX_I$ with $\deg p_I\ge |I|$. 

Consider $w=s_{j}v$ with $\ell(w)>\ell(v)$. Then we obtain
\begin{align*}
\de_w=\de_{s_jv}&=\de_j\sum p_IX_I=\sum_{I}s_j(p_I)\de_jX_I=\sum_{I}s_j(p_I)(1-x_jX_j)X_I\\
&=\sum_{I}s_j(p_I)X_I-\sum_{I}s_j(p_I)x_jX_{\{j\}\cup I}.
\end{align*}
Since $\deg s_j(p_I)\ge |I|$ and $\deg [s_j(p_I)x_j]\ge |I|+1$, the conclusion follows.

(b) From part (a) we can write $\de_v=\sum_{I}p_IX_I$ with $\deg p_I\ge |I|$. For each $I$, we have $X_I=\sum_{\ell(w)\le |I|}p'_{I,w}X_{I_w}$ with $p'_{I,w}\in S$. Therefore, 
\[
\de_v=\sum_{w}\sum_{|I|\ge \ell(w)}p_Ip'_{I,w}X_{I_w},
\]
so $b^X_{v,w}=\sum_{|I|\ge \ell(w)}p_Ip'_{I,w}$. Since 
\[\deg p_Ip'_{I,w}\ge \deg p_I\ge |I|\ge \ell(w),\]
we have $\deg b^X_{v,w}\ge \ell(w)$. 

To prove the last part, consider the map $\DcF^{(0/-1)}\to \nila$ defined in \eqref{eq:FtoFA}, then $\de_v=\sum_{w\le v}b^X_{v,w}X_{I_w}$ is mapped to $\de^A_v=\sum_{w\le v}(-1)^{\ell(w)}(b^X_{v,w})_{\ell(w)}x_w$ in $\nila$ because this map preserves product. Since  $\de_v=\sum_{w\le v}b_{v,w}x_w$ in the twisted group algebra of \cite{KK86} and the set $\{x_w\}_{w\in W}$ is linearly independent, so $b_{v,w}=(-1)^{\ell(w)}(b^X_{v,w})_{\ell(w)}$. 
\end{proof}

\begin{lem}\label{lem:dualfiltration}
(a) The submodule $(\DcFd)^{(i)}$ is (possibly infinitely) spanned by elements $qX_{I_w}^*$ such that $\deg q+\ell(w)\ge i$.

(b) The quotient $(\DcFd)^{(i/i+1)}$ is finitely spanned by elements $qX^*_{I_w}$ such that $\deg q+\ell(w)=i$. 
\end{lem}
\begin{proof} It follows from Lemma \ref{lem:dualfil}.(2)  that $qX^*_{I_w}\in (\DcFd)^{(\deg q+\ell(w))}$. Moreover,  we have
\[qX^*_{I_w}(\de_w)=qb^X_{w,w}=(-1)^{\ell(w)}qc_w\in \IF^{\deg q+\ell(w)}\backslash \IF^{\deg q+\ell(w)+1}. 
\]
So $qX^*_{I_w}\not\in (\DcFd)^{(j)}$ for $j>\deg q+\ell(w)$. 
\end{proof}

The following corollary shows that the filtration agrees with the one defined in \cite[\S7.4]{CPZ}.

\begin{cor} We have $(\DcFd)^{(i)}=\{f\in \DcFd\mid f(\DcF^{(-i+1)})\subseteq \IF\}$.
\end{cor}
\begin{proof}By Lemma~\ref{lem:dualfiltration} $(\DcFd)^{(j)}$ is spanned by elements $qX_{I_w}^*$ with $\deg q+\ell(w)\ge j$. By definition, $\DcF^{(-i+1)}$ is spanned by $pX_{I_v}$ with $\deg p-\ell(v)\ge -i+1$. We have $pX^*_{I_w}(qX_{I_v})=pq\de_{w,v}$. If $w=v$, then $\deg (pq)=\deg p+\deg \ge j-i+1$. So $pX_{I_w}^*(qX_{I_v})\in \IF$ if and only if $j\ge i$.  
\end{proof}

Let $Q_{a,W}$ be the twisted group algebra in Example \ref{prop:addcompareDem} with basis $\{\de_w^a\}$, and $Q_{a,W}^*$ be its dual. Then the $S_R^*(\Lambda)$-dual  $\nila^*$ of $\nila$ is contained in $Q_{a,W}^*$. Since $\{x_w\}_{w\in W}$ is a basis of $\nila$, it defines a dual basis $\{x_w^*\}$ in $\nila^*$.  We define a map of rings
\[
\phi_i:(\DcFd)^{(i)}\to Q_{a,W}^*, \quad \phi_i(f)(\de_w^a)=(f(\de_w))_i,
\]

If $f\in (\DcFd)^{(i+1)}$, then $f(\de_w)\in \IF^{i+1}$, so $(f(\de_w))_i=0$. Therefore, $\phi_i$ induces a map 
\[
\phi_i':(\DcFd)^{(i)}/(\DcFd)^{(i+1)}\to Q_{a,W}^*.
\]

We now ready to state and to prove the main result of this section
\begin{theo}(cf. \cite[Proposition 2.30]{KK90})\label{thm:grnil}
The map $\phi:=\oplus_{i\ge 0}\phi_i'$ is a ring isomorphism $\Gr\DcFd\cong \nila^*$.
\end{theo}
\begin{proof}
First, we show that $\phi$ is a ring homomorphism. Suppose $f\in (\DcFd)^{(i)}/(\DcFd)^{(i+1)}$ and $g\in (\DcFd)^{(j)}/(\DcFd)^{(j+1)}$, then 
\begin{align*}
\phi(fg)(\de_w^a) &=(fg(\de_w))_{i+j}=(f(\de_w))_i(g(\de_w))_j\\
&=\phi(f)(\de_w^a)\cdot \phi(g)(\de^a_w)=[\phi(f)\phi(g)](\de_w^a).
\end{align*}
So $\phi(fg)=\phi(f)\phi(g)$. 

Let $f=\sum'_{w\in W}q_wX_{I_w}^*\in (\DcFd)^{(i)}/(\DcFd)^{(i+1)}$ with $q_w\in S$. Then by Lemma~\ref{lem:dualfiltration}, we know that $\ell(w)=i-\deg q_w\le i$, so it is a finite sum. Therefore, it suffices to show that  $\phi_i'(qX_{I_w}^*)=(-1)^{\ell(w)}(q)_{\deg q}x_w^*$ for any $q\in S$ such that $\deg q+\ell(w)=i$. Recall from Example \ref{prop:addcompareDem} that in $\nila$, we have $x_v=\sum_{u}a_{v,u}\de_u^a$ such that $\sum_{u}a_{v,u}b_{u,w}=\de_{v,w}$. In $\DcF$, by Lemma \ref{lem:dualfil} we have $\de_u=\sum_{w'}b^X_{u,w'}X_{I_{w'}}$ such that $(b^X_{u,w'})_{\ell(w')}=(-1)^{\ell(w')}b_{u,w'}$. So 
\begin{align*}
\phi_i'(qX_{I_w}^*)(x_{v}) &=\phi_i'(qX_{I_w}^*)(\sum_ua_{v,u}\de_u^a)=\sum_{u}a_{v,u}\phi_i(qX_{I_w}^*)(\de_u^a)\\
&=\sum_ua_{v,w}(qX_{I_w}^*(\de_u))_i=\sum_{u}a_{v,u}(qX^*_{I_w}(\sum_{w'}b^X_{u,w'}X_{I_{w'}}))_{i}\\&=\sum_{u}a_{v,u}(qb^X_{u,w})_i=\sum_{u}(q)_{\deg q}a_{v,u}(-1)^{\ell(w)}b_{u,w}\\
&=(-1)^{\ell(w)}(q)_{\deg q}\de_{v,w}.
\end{align*}
Therefore, $\phi'_i(qX^*_{I_w})=(-1)^{\ell(w)}(q)_{\deg q}x_w^*$ and $\oplus_{i\ge 0}\phi_i'\colon \Gr\DcFd\to \nila^*$ is an isomorphism.
\end{proof}

\section{Formal affine Hecke algebras}\label{sec:formhecke}
In this section we define the formal affine Hecke algebra for a generalized Cartan matrix and a Demazure lattice
following \cite{HMSZ} and \cite{ZZ}. We state the structure theorem, whose proof is similar to the one for the formal affine Demazure algebra.

Fix a free abelian group $\Ga$ of rank 1  generated by $\ga$, and let $(S^F)':=R\lbr\Ga\oplus\cl\rbr_F$. Define   $(Q^F)'= (S^F)'[\frac{1}{x_\al }\mid\al\in \Phi]$, and let $(Q_W^F)'$ be the respective twisted formal group algebra defined over $R$, i.e., $(Q^F_W)'=(Q^F)'\rtimes_{R} R[W]$. For each  root $\al$, define the element 
 $$T_\al^F:=x_\ga X_\al+\de_\al\in (Q_W^F)'.$$
If  the formal group law $F$ is clear from the context, we will write $S'=(S^F)'$, $Q'=(Q^F)'$,  $Q_W'=(Q_W^F)'$ and $T_i=T_{\al_i}^F$ for simplicity. We have $\de_wT_{\al_i} \de_{w^{-1}}=T_{w(\al_i)}$. 

For a sequence $I=(i_1,...,i_l)$ in $[n]$, we denote 
$$T_I=T_{i_1...i_l}=T_{i_1}\cdot ...\cdot T_{i_l}.$$
For each $w$, we choose a reduced sequence $I_w$. Then  $T_{I_w}$ depends on the choice of $I_w$ unless $F$ is additive or multiplicative. It is straightforward to show that
\[
T_i^2=x_\ga \ka_iT_i+1-x_\ga \ka_i, \quad T_iq-s_i(q)T_i=x_\ga\Dem_i(q), ~q\in Q.
\]
 \begin{defi} Define the \textit{formal affine Hecke algebra} $\HF$ to be  the $R$-subalgebra of $Q_W'$ generated by elements of $S$ and $T_{i}, i\in [n]$.
 \end{defi}

Note that the definition of $\HF$ depends on the choice of simple roots, since $\de_i\not\in \HF$ and $T_{w(\al_i)}:=\de_wT_i\de_{w^{-1}}$ may not belong to $\HF$. 

\begin{lem} For any reduced sequence $I_w$ we have \[T_{I_w}=\sum_{v\le w}q_{I_w, v}\de_v\text{ with }q_v\in Q',\text{ and }q_w=\prod_{\al\in \Phi_w}\frac{x_\al-x_\ga}{x_\al}.\]
\end{lem}
\begin{proof}Note that $T_\al=\frac{x_\ga}{x_\al}+\frac{x_\al-x_\ga}{x_\al}\de_\al$, so the conclusion follows similar to the proof of \cite[Lemma 5.4]{CZZ1}.
\end{proof}
\begin{lem}\label{lem:keyHecke}Assume that $R$ is a domain of characteristic 0.  
Let $z=\sum_{\ell(w)\le k}p_wT_{I_w}$, $p_w\in S'$, be such that $z\cdot S'\subseteq x_\al S'$. Then $x_\al|p_w$ for all $w$.
\end{lem}
\begin{proof}The proof is similar to that of Lemma~\ref{lem:keyalpha}. The only difference is that, if we write 
\[
z=\sum_{\ell(w)\le k}p_wT_{I_w}=\sum_{\ell(w)\le k}q_w\de_w, ~q_w\in Q',
\]
then $q_w=\frac{p_wd_w}{c_w}$ in the case $\ell(w)=k$, where $d_w:=\prod_{w\in \Phi_w}(x_\al-x_\ga)$. For any root $\al$ and $w\in W$, we have $V(x_\al)\not\subset V(d_w)$, where $V(q)$ denotes set of prime ideals of $S'$ containing $q$. 
So in the proof of Lemma \ref{lem:keyalpha}, we may assume that $\frakp\in V(x_\al)\backslash \cup_{w\in W}V(d_w)$. 
\end{proof}

We then obtain the following analogue of the structure theorem

\begin{theo}Assume that $R$ is a domain of characteristic 0.  Then we have 
\[\HF=\{z=\sum_{w}q_wT_w\in Q_W\mid q_w\in Q' \text{ and }z\cdot S'\subset S'\}.\] Moreover, $\HF$ is a free $S'$-module  with basis $\{T_{I_w}\}_{w\in W}$.
\end{theo}
\begin{proof}
	It follows from \cite[Proposition 6.8]{HMSZ} that each element of $\HF$ can be written as a linear combination $z=\sum_{w\in W}q_wT_{I_w}$ with $q_w\in (Q^F)'$. Moreover, since $T_i\cdot S'\subseteq S'$, so $\HF\cdot S'\subseteq S'$. Therefore, 
\[
\HF\subseteq\{z=\sum_{w\in W}q_wT_{I_w}\mid q_w\in Q',\; z\cdot S'\subseteq S'\}.
\] 
To prove the opposite inclusion,  we replace $S$ by $S'$ and $X_{I_w}$ by $T_{I_w}$ in the proof of Theorem \ref{thm:mainDem}, and use Lemma~\ref{lem:keyHecke}. 

Finally, similar to~Lemma~\ref{lem:deltatoX}, we show that $\{T_{I_w}\}_{w\in W}$ are linearly independent. 
Moreover, Lemma~\ref{lem:keyHecke} also shows that if $z=\sum_wq_wT_{I_w}\in \HF$, then $q_w\in S'$. So $\{T_{I_w}\}$ is a $S'$-basis of $\HF$.
\end{proof}
The proof of the following corollary is similar to the finite case proven in~\cite{Zh}.
\begin{cor}The ring homomorphism $\HF\to \End_{R_F}(S')$ is injective, and the centre of $\HF$ is $(S')^W$.
\end{cor}

\section{Formal root polynomials and the Pieri-type formula}

In the present section we generalize the notion of a formal root polynomial introduced in \cite{LZ} and discuss the multiplication rule
in $\DcFd$.

Let $\cl$ be a Demazure lattice and let $S$ and $Q$ be the associated formal group algebra and its localization.
Following \cite[\S2]{LZ} consider the ring $Q_W[[\cl]]:=S\otimes_R Q_W$, where the elements of $S$ on the left (denoted by $y$'s) commute with the elements of $Q_W$. Given $w\in W$ and a reduced word $I_w=s_{i_1}\ldots s_{i_l}$ we set
\[
\mathcal{R}_{I_w}^X:=\prod_{k=1}^l h_{i_k}^X(s_{i_1}\ldots s_{i_{k-1}}(\al_{i_k})),\text{ where }h_i^X(\la)=1-y_{-\la}X_i, \text{ and }
\]
\[
\mathcal{R}_{I_w}^Y:=\prod_{k=1}^l h_{i_k}^Y(s_{i_1}\ldots s_{i_{k-1}}(\al_{i_k})),\text{ where }h_i^Y(\la)=1-y_{\la}Y_i.
\]

Consider the evaluation map $ev\colon Q_W[[\cl]] \to Q_W$ induced by $y_\la \mapsto x_{-\la}$. Then by the same proof as of \cite[Lemma~3.3]{LZ}) we obtain
\[
ev(\mathcal{R}_{I_w}^X)=\delta_w\text{ and }ev(\mathcal{R}_{I_w}^Y)=\theta_{w}\delta_w,\]
where
\[
\theta_{w}=\prod_{\al\in \Sigma_w} \theta(\al),\quad \theta(\la)=\tfrac{-x_{-\la}}{x_\la}.
\]
Moreover, by the same arguments as in \cite[\S3.2]{LZ} we obtain that $\mathcal{R}_{I_w}$ is independent of a choice of the reduced word $I_w$ iff the formal group law $F$ is hyperbolic  (i.e., $F(x,y)=\frac{x+y-\mu_1xy}{1+\mu_2xy}$ and $R=\Z[\mu_1,\mu_2]$) and $2$ is regular in $R$.

Indeed, if the Coxeter exponent $m_{ij}= 2,3,4,6$, we apply the same arguments as in \cite[Prop~3.2]{LZ}.
If $m_{ij}$ is not from that list, then there is no relation between $s_i$ and $s_j$ which does not affect $\mathcal{R}_{I_w}$.
The elements $\mathcal{R}_w^X=\mathcal{R}_{I_w}^X$ and $\mathcal{R}_w^Y=\mathcal{R}_{I_w}^Y$ of $Q_W[[\cl]]$ will be called formal root polynomials.

As in \cite[\S4]{LZ} the coefficients $K^X(I_v,w)$  (resp. $K^Y(I_v,w)$) in the expansions
\[
\mathcal{R}_w^Y = \sum_{v\le w} K^Y(I_v,w) Y_{I_w},\quad \mathcal{R}_w^X = \sum_{v\le w} K^X(I_v,w) X_{I_w}
\]
satisfy the following 
\begin{lem}(cf. \cite[Thm.~4.4]{LZ})\label{lem:hyp} Suppose that $F$ is a hyperbolic formal group law. We have 
\[
b^{Y}_{w,I_v}=\tau(\theta_w ev(K^Y(I_v,w)))\text{ and }b^X_{w,I_v}=\tau(ev(K^X(I_v,w))),
\]
where $\tau$ is an involution $x_\la\mapsto x_{-\la}$ and $b_{w,I_v}$ are coefficients in the expansion:
\[
\delta_w=\sum_{v\le w}b_{w,I_v}^YY_{I_v}=\sum_{v\le w}b_{w,I_v}^X X_{I_v}.
\]
\end{lem}

Similar to \cite[Prop. 4.7]{LZ} we obtain the following formula for the multiplication:
\begin{lem} We have in $\DcFd$
\[
Y_{I_u}^* \cdot Y_{I_v}^* = \sum_{w\ge u,v} p^{I_w}_{I_u,I_v} Y_{I_w}^*,
\]
where the coefficients $p^{I_w}_{I_u,I_v}$ satisfy
\begin{equation}\label{eq:rec}
p^{I_w}_{I_u,I_v}=\tfrac{1}{b^Y_{w,I_w}}(b^Y_{w,I_u}b^Y_{w,I_v}-\sum_{u,v\le t<w} p^{I_t}_{I_u,I_v} b^Y_{w,I_t}).
\end{equation}
\end{lem}
Indeed, $p^{I_w}_{I_u,I_v}=0$ if either $u$ is not a subword of $w$ or $v$ is not a subword of $w$.
Therefore, we can compute all the coefficients recursively starting with $w=u$ (for $v\le u$, $p_{I_u,I_v}^{u}=b_{u,I_v}^Y$) and going up in length of $w$. Moreover, by Lemma~\ref{lem:hyp} one can find the coefficients $b_{u,I_v}^Y$ by expanding the formal root polynomial in the hyperbolic case.

Observe that the same recurrent formula holds with $Y$'s replaced by $X$'s.

\begin{rem}
For an additive $F$ 
there is a closed formula for $b_{u,v}^Y$ given by Billey \cite[Theorem~4]{Bi99} (see also \cite[(47)]{LZ})
\[
b_{u,v}^Y=\sum_{1\le j_1<\ldots <j_k \le l} \beta_{j_1}\ldots \beta_{j_k},
\]
where $I_u=(i_1,\ldots,i_l)$, $\beta_j=s_{i_1}\ldots s_{i_{j-1}}(\al_{i_j})$, $k=\ell(v)$ and the summation
ranges over the integer sequences for which the subword $(i_{j_1},\ldots,i_{j_k})$ is a reduce word for $v$.
For a multiplicative $F$ there are closed formulas given by Graham-Willems (see \cite[\S5.2]{LZ}).
\end{rem}

\begin{example} 
Let $F$ be the hyperbolic formal group law.  Using the formula~\eqref{eq:rec} we find $p_{e,e}^e=1$, $p_{s_i,e}^{s_i}=b^Y_{s_i,e}=\tfrac{-x_i}{x_{-i}}$, $p_{s_i,s_i}^{s_i}=b^Y_{s_i,s_i}=x_i$ and 
\[
p_{e,e}^{s_i}=\tfrac{1}{x_i}\big((\tfrac{-x_i}{x_{-i}})^2- \tfrac{-x_i}{x_{-i}}\big)=\tfrac{x_i}{x_{-i}}\kappa_i \in S,\quad \text{where}\; \kappa_i=\tfrac{1}{x_i}+\tfrac{1}{x_{-i}}.
\]
Combining we obtain
\[
Y_e^* \cdot Y_e^*=Y_e^*+\sum_{i} (\tfrac{x_i}{x_{-i}}\kappa_i) Y_{i}^*+\sum_{ \ell(w)\ge 2} p_{e,e}^w Y_{I_w}^*\text{ and}
\]
\[
Y_i^* \cdot Y_{e}^*=(\tfrac{-x_i}{x_{-i}}) Y_i^*+ \sum_{w> s_i} p_{s_i,e}^w Y_{I_w}^*.
\]

Observe that for a general formal group law $Y_e^*\neq \mathbf{1}$ in $\DcFd$. In the finite case, $Y_e^*$ is the class of the Bott-Samelson resolution
of $G/B$ corresponding to some reduced expression of the element of maximal length (it does not necessarily coincides with the fundamental class).
For the usual cohomology (in the finite case) $Y_{I_w}^*$ are the Poincar\'e duals to the usual Schubert classes. In particular, $Y_i^*$ corresponds to the class of a divisor.
\end{example}

\begin{example}
Consider $s_js_i$ with $j\neq i$. Set $x_{j(i)}:=x_{s_j(\alpha_i)}$.
Then $b^Y_{w,w}=x_jx_{j(i)}$, $b^Y_{w,s_i}=-\tfrac{x_jx_{j(i)}}{x_{-j}}$, $b^Y_{w,s_j}=-\tfrac{x_jx_{j(i)}}{x_{-j(i)}}$ and $b^Y_{w,e}=\tfrac{x_jx_{j(i)}}{x_{-j}x_{-j(i)}}$.
Using the formula~\eqref{eq:rec} we obtain
\begin{gather*}
p_{s_i,s_i}^{s_js_i}=\tfrac{1}{b^Y_{w,w}}\big( (b^Y_{w,s_i})^2 - b^Y_{s_i,s_i}b^Y_{w,s_i}\big)=\tfrac{1}{x_{-j}}(\tfrac{x_jx_{j(i)}}{x_{-j}}+x_i),\\
p_{s_i,s_j}^{s_js_i}=p_{s_j,s_i}^{s_js_i}=\tfrac{1}{b^Y_{w,w}}b^Y_{w,s_i}b^Y_{w,s_j}=\tfrac{x_jx_{j(i)}}{x_{-j}x_{-j(i)}},\\
p_{s_j,s_j}^{s_js_i}=\tfrac{1}{b^Y_{w,w}}\big( (b^Y_{w,s_j})^2 - b^Y_{s_j,s_j}b^Y_{w,s_j}\big)=\tfrac{1}{x_{-j(i)}}(\tfrac{x_jx_{j(i)}}{x_{-j(i)}}+x_j).
\end{gather*}
Therefore, we obtain in $\DcFd$
\begin{align*}
Y_i^*\cdot Y_i^*=x_iY_i^*+& \sum_{s_is_j=s_js_i} x_i\tfrac{x_j}{x_{-j}}\kappa_j Y_{s_js_i}^*\\
+& \sum_{s_is_j\neq s_js_i}  [x_i\tfrac{x_{i(j)}}{x_{-i(j)}}\kappa_{i(j)}Y_{s_is_j}^*+\tfrac{1}{x_{-j}}(\tfrac{x_jx_{j(i)}}{x_{-j}}+x_i)Y_{s_js_i}^*]\\
+& \sum_{\ell(w)\ge 3} p_{s_i,s_i}^{w}Y_{I_w}^*.
\end{align*}
In particular, for an additive $F$ we obtain in $\nila^*$
\[
Y_i^*\cdot Y_i^* = \alpha_i Y_i^* + \sum_{s_is_j\neq s_js_i} (-\al_j^\vee(\al_i))Y_{s_js_i}^*.
\]
Observe that the general formula for a multiplication by $Y_i^*$ in the additive case was obtained in \cite[(4.30)]{KK86}.
\end{example}

\end{document}